\RequirePackage[l2tabu, orthodox]{nag}
\RequirePackage{fixltx2e}
\documentclass[10pt,a4paper,reqno,oneside,final]{smfart}

\usepackage{leftidx}
\usepackage{appendix}
\usepackage{tikz}\usetikzlibrary{decorations.markings,matrix,arrows}
\usepackage{amsfonts}
\usepackage{amsthm}
\usepackage[T1]{fontenc}
\usepackage[mathscr]{eucal}
\usepackage{setspace}
\usepackage{amssymb,latexsym,amsmath,amscd}
\usepackage{graphicx}
\usepackage{showkeys}
\usepackage[french]{babel}
\usepackage{layout}
\usepackage{enumerate}
\usepackage{indentfirst}
\usepackage[varg]{pxfonts}
\usepackage[stretch=10]{microtype}
\usepackage{booktabs}
\usepackage{xspace}
\usepackage{eufrak}
\usepackage[colorlinks=false, pdfborder={0 0 0}]{hyperref} 
\usepackage{nameref}
\usepackage{cleveref}
\usepackage{calrsfs}
\usepackage{filecontents}
\usepackage{mathtools}
\usepackage{titlesec}
\usepackage{fixltx2e}
\usepackage{todonotes}

\titleformat{\paragraph}
{\normalfont\normalsize\bfseries}{\theparagraph}{1em}{}
\titlespacing*{\paragraph}
{0pt}{3.25ex plus 1ex minus .2ex}{1.5ex plus .2ex}

\newtheoremstyle{exostyle} 
{\topsep}
{\topsep}
{}
{}
{\bfseries}
{.}
{ }
{\thmname{#1}\thmnumber{ #2}\thmnote{. \normalfont{\textit{#3}}}}

\theoremstyle{exostyle} 
 
\makeatletter
\newcommand{\neutralize}[1]{\expandafter\let\csname c@#1\endcsname\count@}
\makeatother

\theoremstyle{plain}
\newtheorem{thm}{Theorem}[section]

\newtheorem*{thm*}{Theorem}

\newtheorem{lem}[thm]{Lemma}

\newtheorem{pro-def}[thm]{Proposition-Definition}
\newtheorem{cor}[thm]{Corollary}
\newtheorem{conj}[thm]{Conjecture}

\theoremstyle{definition}

\newtheorem{rem}[thm]{Remark}

\theoremstyle{remark}

\newcommand{\ssec}{\subsection}

\newcommand{\wt}{\widetilde}

\newcommand{\bP}{\mathbf{P}}

\newcommand{\bC}{\mathbf{C}}

\newcommand{\bR}{\mathbf{R}}

\newcommand{\bZ}{\mathbf{Z}}

\newcommand{\gO}{\Omega}

\newcommand{\gS}{\Sigma}

\newcommand{\cD}{\mathcal{D}}
\newcommand{\cC}{\mathcal{C}}

\newcommand{\cJ}{\mathcal{J}}

\newcommand{\cP}{\mathcal{P}}

\newcommand{\cO}{\mathcal{O}}

\newcommand{\fS}{\mathfrak{S}}

\newcommand{\ol}{\overline}

\newcommand{\colonec}{\mathrel{:=}}

\newcommand{\alb}{\mathrm{alb}}

\newcommand{\Id}{\mathrm{Id}}

\newcommand{\CH}{\mathrm{CH}}

\newcommand{\Prym}{\mathrm{Prym}}

\newcommand{\Pic}{\mathrm{Pic}}

\newcommand{\Ker}{\mathrm{Ker}}

\newcommand{\Ima}{\mathrm{Im}}

\renewcommand{\(}{\left(}
\renewcommand{\)}{\right)}

\newcommand{\dto}{\dashrightarrow}

\newcommand{\vast}{\bBigg@{4}}
\newcommand{\Vast}{\bBigg@{5}}

\makeatletter
\let\orgdescriptionlabel\descriptionlabel
\renewcommand*{\descriptionlabel}[1]{%
  \let\orglabel\label
  \let\label\@gobble
  \phantomsection
  \edef\@currentlabel{#1}%
  \let\label\orglabel
  \orgdescriptionlabel{#1}%
}
\makeatother
\tikzset{node distance=2cm, auto}

\numberwithin{equation}{section}

\title{Rational maps from punctual Hilbert schemes of K3 surfaces} 

\author{Hsueh-Yung Lin}
\address{Centre de Mathématiques Laurent Schwartz, 91128 Palaiseau Cédex, France}

\begin{document}

\maketitle

\begin{abstract}
The purpose of this short note is to study dominant rational maps from  punctual Hilbert schemes of length $k\ge2$ of  projective $K3$ surfaces. Precisely, we  prove that their image is necessarily rationally connected if the rational map is not generically finite. As an application, we simplify C. Voisin's proof of the fact that symplectic involutions of any projective $K3$ surface $S$ act trivially on $\CH_0(S)$.
\end{abstract}

\section{Introduction}



In this note, we will work throughout over the field of complex numbers $\bC$. Recall that a $K3$ surface $S$ is by definition a smooth projective surface with trivial canonical bundle $K_S = \gO_S^2$ and vanishing $H^1(S, \cO_S)$. The Hilbert scheme of zero-dimensional subschemes of length $k \ge 2$ on the $K3$ surface $S$ will be denoted by $S^{[k]}$.  

Recall that a proper variety $X$ is said to be \emph{uniruled} (resp. \emph{rationally connected}) if a general point $x \in X$ (resp. two general points $x,y \in X$) is contained in the image of a non-constant map $\bP^1 \to X$. These are obviously birationally invariant properties. It is also clear that $\CH_0(X) = \bZ$ for any rationally connected variety $X$. When $X$ is smooth, rational connectedness is equivalent to the \emph{a priori} weaker condition that  two general points can be joined by a chain of rational curves~\cite[Theorem $2.1$]{KMM}.

The following is the main result we obtain in this article:

\begin{thm}\label{RCbase}
If $f: S^{[k]} \dto B$ is a dominant rational map to a variety $B$ with $\dim B < \dim S^{[k]}$, then either $B$ is a point or rationally connected.
\end{thm}
One can compare this theorem with a result of Matsushita~\cite{MatsushitaFib}, showing that for any surjective morphism from an irreducible symplectic manifold to a projective base $B$ of positive dimension, if the general fiber is of positive dimension, then $B$ is Fano. As an application of Theorem~\ref{RCbase}, we will give an alternative proof of C. Voisin's main result in~\cite{Voisinsyminv}.



\begin{thm}[Voisin]\label{syminv}
Suppose $S$ is a projective $K3$ surface and $\imath$ is a symplectic involution acting on $S$, then $\imath$ acts as the identity on $\CH_0(S)$.
\end{thm}

The motivation for the statement of Theorem~\ref{syminv} comes from the following conjecture, which is a consequence of the generalized Bloch conjecture for surfaces~\cite[Partie VII]{Voisin}:

\begin{conj}\label{blochfinite}
If $S$ is a surface with $q= h^{0,1} = 0$ and $f:S \to S$ is an automorphism of finite order acting trivially on $H^0(S, \gO_S^2)$, then the induced map $f_*$ acts as the identity on $\CH_0(S)$.
\end{conj}

A series of examples of surfaces with $q=0$ is provided by $K3$ surfaces $S$. Such  surfaces have one-dimensional $H^0(S, \gO_S^2)$ generated by a non-degenerated holomorphic two-form $\eta$. An automorphism $f:S \to S$ such that $f^*\eta = \eta$ is called a \emph{symplectic automorphism}.

A  recent new advance  of Conjecture~\ref{blochfinite} for $K3$ surfaces was made by D. Huybrechts and M. Kemeny in~\cite{HuyKem}. They worked with invariant elliptic curves and  solved Conjecture~\ref{blochfinite} for $K3$ surfaces with symplectic involutions $f$ in one of the three series in the classification introduced by van Geemen and Sarti~\cite{GeeSar}. In~\cite{Voisinsyminv}, C. Voisin  showed in general  that symplectic involutions act trivially on $\CH_0(S)$ for any projective $K3$ surface $S$. The general statement of Conjecture~\ref{blochfinite} for $K3$ surfaces is proved soon after in~\cite{Huybfini} by D. Huybrechts.

\begin{thm}[Huybrechts, Voisin]\label{blochfiniteK3}
Let $S$ be a projective $K3$ surface, $\eta$ be a non-zero holomorphic two-form on $S$, and $f: S\to S$ be a symplectic automorphism of finite order on $S$, then $f$ acts  trivially on $\CH_0(S)$.
\end{thm}

As was shown by Nikulin in~\cite{Nikulin}, the only possible orders of $f$ range from one to eight. In Huybrechts' proof, he studied case by case according to these finitely many possible orders using  derived technique and Garbagnati and Sarti's classification results~\cite{GarSar} on lattices of the invariant part $H^2(X,\bZ)^f$ of the action of symplectic automorphism $f$ with prime order.

The main construction in Voisin's proof~\cite{Voisinsyminv} of Theorem~\ref{syminv} is the factorization 
\begin{equation}\label{factor}
\begin{tikzpicture}
\centering
\matrix (m) [matrix of math nodes, row sep=1.5em,
column sep=1.5em, text height=1.5ex, text depth=0.25ex]
{  S^{[g]} &  & \CH_0(S) \\
 & \cP_{(S,H)}(\wt{\cC} / \cC) &    \\};
\path[densely dashed, ->,font=\scriptsize]
(m-1-1) edge node[below left] {$\gamma_S$} (m-2-2);
\path[ ->,font=\scriptsize]
(m-1-1) edge node[auto] {$\Gamma_*$} (m-1-3)
(m-2-2) edge node[auto] {} (m-1-3);
\end{tikzpicture}
\end{equation} 
where $\cP_{(S,H)}(\wt{\cC} / \cC)$ is the universal Prym variety associated to a complete linear system of curves of genus $g$ in the quotient surface $S / \imath$. Details of the above construction will be given in Section~\ref{syminvv}. Our main application of Theorem~\ref{RCbase} is the following

\begin{thm}\label{prymRC}
Any smooth projective compactification of $\cP_{(S,H)}(\wt{\cC} / \cC)$ is rationally connected.
\end{thm}

The organization of this paper is as follows. In Section~\ref{equiv}, we will recall some well-known facts concerning rationally connected varieties and use them to reformulate Theorem~\ref{RCbase}. Then we will prove Theorem~\ref{RCbase} in Section~\ref{proof1}. In  Section~\ref{syminvv}, we will explain how Theorem~\ref{prymRC} gives an alternative proof of Theorem~\ref{syminv}.

\section*{Acknowledgement}
I would like to thank my supervisor Claire Voisin for helpful discussions, her kindness in sharing her ideas, and her patience during this work. I also thank De-Qi Zhang for the simplification of the original proof of Theorem~\ref{RCbase}.

\section{Remarks on rationally connected varieties}\label{equiv}

In this section, we first recall the definition of MRC-fibrations and introduce the main theorem of Graber-Harris-Starr in~\cite{GHS}. Then we will use them to give a reformulation of Theorem~\ref{RCbase}. 

Recall from~\cite{KMM} that for any variety $X$, there exists a rational map $\phi : X \dto B$ unique up to birational equivalence characterized by the following properties:
\begin{enumerate}[(i)]
\item a general fiber of $\phi$ is rationally connected;
\item for a  general point $b\in B$, any rational curve passing through $X_b \colonec \phi^{-1}(b)$ is actually contained in $X_b$.
\end{enumerate}

The map $\phi : X \dto B$ is called the \emph{maximal rationally connected} (or MRC for short) \emph{fibration of $X$}. The fundamental question  whether the base $B$ of an MRC-fibration is not uniruled remained open for  a while~\cite[Conjecture IV.5.6]{Kollarrat}. It was finally answered by T. Graber, J. Harris, and J. Starr as a direct corollary of their main theorem in their  paper~\cite{GHS}: 

\begin{thm}[Graber-Harris-Starr]\label{GHS}
Let $g: X \to C$ be a proper morphism between complex varieties where $C$ is a smooth curve. If the general fiber of $g$ is rationally connected, then $g$ has a section.
\end{thm}

\begin{rem}
This theorem was later generalized by Starr and de Jong to varieties defined over an arbitrary algebraically closed field: any proper morphism from a smooth variety to a smooth curve whose general fibers are smooth and separably rationally connected has a section~\cite{deJongStarr}.
\end{rem}

In particular, Theorem~\ref{GHS} implies:

\begin{cor}[Graber-Harris-Starr~\cite{GHS}]\label{mrc}
Let $g: X \dto Z$ be a maximal rationally connected fibration where $X$ is an irreducible variety, then $Z$ is not uniruled. 
\end{cor}
Let us recall for completeness the proof of the above corollary.

\begin{proof}[Proof]
After resolving the rational map $g$ and singularities of $X$, we can assume that  $g$ is a morphism and $X$ is smooth. Suppose that $Z$ were uniruled, then there exists a rational curve $C$ passing through a general point of $Z$ and we can suppose that $g$ is dominant on $C$. Denote by $\wt{C}$ the normalization of $C$. Up to replacing  $X \times_Z \tilde{C}$ by its desingularization, the map $ g_C : X \times_Z \tilde{C}  \to \tilde{C}$ has a section $D$ by Theorem~\ref{GHS}, which is also a rational curve. Thus if $y \in D$, the rational curve $D$ passes through $y$ and is not contracted by $g$, which is absurd because $g: X \to Z$ is an MRC-fibration.
\end{proof}

\begin{rem}
An equivalent formulation of Corollary~\ref{mrc}  is the following: if $f: X \dto B$ is a dominant map such that both the general fibers of $f$ and the base $B$ are rationally connected, then $X$ is rationally connected as well~\cite[Proposition IV.5.6.3]{Kollarrat}.
\end{rem}

Thanks to Corollary~\ref{mrc}, we can easily show that

\begin{cor}\label{equi}
In Theorem~\ref{RCbase}, it is equivalent to show that $B$ is either a point or uniruled.
\end{cor}

\begin{proof}[Proof] Only one direction needs to be proved. Assume that Theorem~\ref{RCbase} is true when replacing "rationally connected" with "uniruled". Suppose $B$ is not a point; let $B \dto B'$ be an MRC fibration of $B$. Since $B'$ is not uniruled by Corollary~\ref{mrc}, we deduce that $B'$ is a point by assumption. Hence $B$ is rationally connected.
\end{proof}





\section{Proof of Theorem~\ref{RCbase}}\label{proof1}

Let us first prove the following general result.
\begin{lem} \label{dom}
Let $f : X \dto B$ be a rational dominant map between smooth projective varieties such that $\dim X > \dim B$. If $D$ is an ample divisor on $X$, then the restriction map $f_{|D}$ is still dominant provided $B$ is not uniruled.
\end{lem}

\begin{rem}
The following example shows that the hypothesis of $B$ not being uniruled is essential in the above lemma. Let $l$ be a line in $\bP^2$ and $p$ a point in $\bP^2$ which is not contained in $l$. If $f : \bP^2 \dto l$ is the projection  from $p$, then $f$ is dominant and every line in $\bP^2$ containing $p$ is contracted to a point under $f$. 
\end{rem}

\begin{rem}\label{birat}
Before we start the proof, we remark that any birational modification of the rational map $f : S^{[k]} \dto B$ and the base $B$ will not affect the hypotheses and the conclusion in Theorem~\ref{RCbase}. For instance up to desingularization of $B$, we can always suppose that $B$ is smooth and complete. This kind of modification will be repeatedly used in the proof of Theorem~\ref{RCbase}. 
\end{rem}

\begin{proof}[Proof of Lemma~\ref{dom}]
Let $\tilde{f} : \wt{X} \to B$ be a resolution of $f$ after a sequence of blow-ups $\pi : \wt{X} \to X$ of $X$. Denote by $\wt{D}$ the proper transform of $D$ and $D' \colonec \pi^{-1}(D)$. Since $D$ is ample and $[D'] = \pi^*[D]$ in $\Pic(\wt{X})$, 
one deduces that $D'$ is a big divisor. Suppose $k = \dim X - \dim B$ and denoted by $H \in  NS(\wt{X})$ the class of an ample divisor on $\wt{X}$. As $\tilde{f}$ is dominant, the class $\tilde{f}^*[x] \cdot H^{k-1} \in  H_2(X,\bR)$, where $x$ is a point of $B$, is a class of a movable curve. So $[D'] \cdot \tilde{f}^*[x] \cdot H^{k-1} >0$~\cite[Corollary $2.5$]{BDPP}, hence the restriction $\tilde{f}_{|D'}$ is  dominant.

Now let $y$ be a very general point of $B$ such that there is no rational curve passing through $y$.  We know that there exists $x \in D'$  such that $  y = \tilde{f}(x)$. As $\pi(x) \in D$, the fiber $E_x \colonec \pi^{-1}(\pi(x))$ intersects $\wt{D}$.  On the other hand, this fiber is connected and rationally connected.  As there is no rational curve passing through $y$, the fiber $E_x$ is contracted to $y$ by $\tilde{f}$. As $E_x$ meets $\wt{D}$, we conclude that $\tilde{f}^{-1}(y) \cap \wt{D} \ne \emptyset$, so $\tilde{f}_{|\wt{D}}$ (hence $f_{|D}$) is dominant.
\end{proof}

\begin{proof}[Proof of Theorem~\ref{RCbase}]

Assuming $B$ is not uniruled, by Corollary~\ref{equi} it suffices to show that $B$ is a point.
Using the natural map $S^k \dto S^{[k]}$, instead of dealing with $S^{[k]} \dto B$ it is equivalent to look at the maps $S^k \dto B$ which are invariant under the action of $\fS_k$ on $S^k$ by permutation. 

Let $D$ be a rational curve lying in an ample linear system of $S$ (whose existence is due to Bogomolov-Mumford~\cite{MM}). Since $\cO(D)^{\boxtimes k}$ is ample on $S^k$, we deduce by Lemma~\ref{dom} that the restriction of $f$ to $\cup_{i=1}^k S^{i-1} \times D \times S^{k-i}$ (with $S^0 \times D \times S^{k-1} = D \times S^{k-1}$ and $S^{k-1} \times D \times S^{0} = S^{k-1} \times D$) is dominant. As this union is finite, we can suppose without loss of generality that the restriction to $D \times S^{k-1}$ of $f$ is dominant. (In fact, since $f$ is symmetric, this is always the case.) 




Since $D \times S^{k-1} \dto B$ is dominant, for a general point $z \in S^{k-1}$ the rational curve $D \times z$ is contracted to a point by $f$ (whenever defined). Up to birational equivalence of $B$, we can suppose that $D \times z$ is contracted to a point for every point $z \in S^{k-1}$ by Remark~\ref{birat}. Since $D \times z$ is ample in $S \times z$, the fibers of the restriction map $f_z \colonec f_{|S \times z}$ have positive dimension (which is a consequence of Hodge index theorem). So  either $S \times z$ is contracted to a curve $C_z$ or to a point. 

Next, suppose that $S \times z$ is contracted to a curve $C_z$.  Since $0 = q(S) \ge q(C_z)$, $C_z$ is necessarily a rational curve. 
Set
$$U \colonec \bigcup_{S \times z \text{ contracted to a curve}} S\times z \subset S \times S^{k-1}.$$
Since by assumption $B$ is not uniruled, the restriction of $f$ to $U$ is not dominant. Therefore up to birational modification of $B$, we can suppose that $f$ contracts $S \times z$ to a point for any $z \in S^{k-1}$. As $f$ is symmetric, we deduce that for all $0 < i \le k$, the map $f$ contracts $z \times S \times z'$ to a point for any $z \in S^{i-1}$ and $z' \in S^{k-i}$. Therefore the image of $f$ is a point, and we are done. 

\end{proof}


\section{Triviality of symplectic involution actions on $\CH_0(S)$}\label{syminvv}

The aim of this section is to give an alternative proof of Theorem~\ref{syminv} that symplectic involutions of a $K3$ surface $S$ act as the identity map on $\CH_0(S)$ using Theorem~\ref{RCbase}. 

Let $\imath$ be a symplectic involution of $S$ and let $\Gamma = \Delta_S - \Gamma_\imath \in \CH^2(S \times S)$ where $\Gamma_\imath \in S \times S$ is the graph of $\imath$. Before we start the proof, let us recall the factorization of $\Gamma_* : S^{[g]} \to \CH_0(S)$ constructed by Voisin in~\cite{Voisinsyminv}, which is used in an essential way both in Voisin's proof and ours. 

\subsection{Prym varieties and a factorization of $\Gamma_* : S^{[g]} \to \CH_0(S)$}
 
Let $\pi : \wt{C} \to C$ be an étale double cover of a smooth curve $C$ and consider the involution $i:\wt{C} \to \wt{C}$ that interchanges the preimages of any point $p \in C$. This involution $i$ induces an endomorphism on the Jacobian of $\wt{C}$ denoted $i_* : J(\wt{C}) \to J(\wt{C})$, and the Prym variety of $\wt{C} \to C$  is defined as 
$$ \Prym_{\wt{C}/C} = \Ima\(\Id - i_*\) = \Ker\(\Id + i_*\)^\circ,$$
where for any algebraic group $G$, $G^{\circ}$ denotes the connected component containing the identity of $G$.  It is an abelian variety carrying a principal polarization and it is easy to see that $ \Prym_{\wt{C}/C}$ is also isomorphic to $\Ker \(\pi_*\)^\circ$, where $\pi_*$ is the norm map $\pi_* : J(\wt{C}) \to J({C})$. Using that $\pi_*$ is surjective and the Riemann-Hurwitz formula, we deduce that $\dim \Prym_{\wt{C}/C} = g -1$ where $g$ is the genus of $C$.

Now let $\Sigma \colonec S / \imath$ be the (singular) quotient surface of $S$ by the involution $\imath$. Choose a very ample line bundle $H \in \Pic(\Sigma)$ and assume that $c_1(H)^2 = 2g-2$. Since the canonical line bundle $K_\Sigma$ is trivial, the genus of smooth curves in $|H|$ is $g$ and $h^0(H) = g+1$. 

Let $U \subset S^{[g]}$ be a Zariski open subset parametrizing reduced subschemes $s = s_1 + \ldots +s_g \in S^{[g]}$ such that there exists a unique smooth curve $C_s$ in $|H|$ passing through the image of each $s_i$ in $\gS$. Since $C_s$ is a smooth curve, its inverse image $\wt{C_s} \subset S$ is smooth, connected, and is an étale cover of $C_s$, which contains $s_1, \ldots, s_g$. One notices that $\Gamma_*(s) = \sum_i \([s_i] - \imath_*[s_i]\) \in \CH_0(S)$ does not depend on $\alb_{\wt{C_s}}\(\sum_i \([s_i] - \imath_*[s_i]\)\) \in  \Prym_{\wt{C_s}/C_s}$, thus we obtain the following factorization of ${\Gamma_{|U}}_*: U \to \CH_0(S)$:
\begin{equation}\label{factor}
\begin{tikzpicture}
\centering
\matrix (m) [matrix of math nodes, row sep=1.5em,
column sep=1.5em, text height=1.5ex, text depth=0.25ex]
{  U &  & \CH_0(S) \\
 & \cP_{(S,H)}(\wt{\cC} / \cC) &    \\};
\path[->,font=\scriptsize]
(m-1-1) edge node[below left] {$\gamma$} (m-2-2);
\path[ ->,font=\scriptsize]
(m-1-1) edge node[auto] {${\Gamma_{|U}}_*$} (m-1-3)
(m-2-2) edge node[auto] {$p$} (m-1-3);
\end{tikzpicture}
\end{equation}
where $\cC \to U' \subset |H|$ (resp. $\wt{\cC} \to U'$) is the universal smooth curve over the Zariski open set $U'$ of $|H|$ parametrizing smooth curves (resp. universal family of double coverings over $U'$), $\cP_{(S,H)}(\wt{\cC} / \cC)$ is the corresponding universal Prym varieties over $U'$, and $\gamma$ is defined as 
$$\gamma(s) = \alb_{\wt{C_s}}\(\sum_i \([s_i] - \imath_*[s_i]\)\).$$ 

\ssec{Proof of Theorem~\ref{syminv}} 

With the same notations introduced in the last paragraph, let us first make the factorization~\ref{factor} more precise. Let $I \in \CH^2(S^{[g]}\times S)$ be the class of the incidence correspondence. Notice that $\Gamma_* : S^{[g]} \to \CH_0(S)$ factors through 
$$\Gamma_* : \CH_0(S^{[g]}) \to \CH_0(S),$$ where $\Gamma \colonec I-(\Id_{S^{[g]}},\imath)(I) \in \CH^2(S^{[g]} \times S)$.  Let $\ol{\cP}$ be a smooth compactification of $\cP_{(S,H)}(\wt{\cC} / \cC)$. Let
\begin{equation}\label{factor}
\begin{tikzpicture}
\centering
\matrix (m) [matrix of math nodes, row sep=1.5em,
column sep=1.5em, text height=1.5ex, text depth=0.25ex]
{  & V & \\
S^{[g]} &  & \ol{\cP}   \\};
\path[->,font=\scriptsize]
(m-1-2) edge node[above left] {$p$} (m-2-1)
(m-1-2) edge node[auto] {$q$} (m-2-3);
\path[ dashed, ->,font=\scriptsize]
(m-2-1) edge node[auto] {$\gamma$} (m-2-3);
\end{tikzpicture}
\end{equation}
be a resolution of $\gamma$ defined in the previous paragraph. Since Chow groups of zero-cycles are invariant under birational modifications, the rational map $\gamma$ defines canonically the pushfoward map by $\gamma_* \colonec q_*p^* : \CH_0(S^{[g]}) \to \CH_0(\ol{\cP})$.

\begin{lem}\label{fact2}
There exists a codimension $2$ correspondence $\Gamma' \in \CH^2(\ol{\cP} \times S)$ such that
$$\Gamma'_* \circ {\gamma}_* = {\Gamma}_* : \CH_0(S^{[g]}) \to \CH_0(S).$$
\end{lem}

\begin{proof}[Proof]
From the definition of ${\gamma}$, it suffices to show that there exists $\Gamma' \in \CH^2(\ol{\cP} \times S)$ such that the morphism $p : \cP_{(S,H)}(\wt{\cC} / \cC) \to \CH_0(S)$  introduced in~\ref{factor} factors through $\Gamma'_* : \CH_0(\ol{\cP}) \to \CH_0(S)$. Let $\cD$ be  the restriction to $\cP_{(S,H)}(\wt{\cC} / \cC) \times_U \wt{\cC}$ of the universal Poincaré divisor (unique up to linear equivalence) in ${\cJ}ac\(\wt{\cC} / U\)$.  The inclusions of each fiber of $\wt{\cC} \to U$ in $S$ define a map $\phi : \cP_{(S,H)}(\wt{\cC} / \cC) \times_U \wt{\cC} \to \cP_{(S,H)}(\wt{\cC} / \cC) \times S$. Let $\Gamma'$ be the closure of $\phi(\cD)$ in $\ol{\cP} \times S$, then the induced correspondence $\Gamma'_* : \CH_0(\ol{\cP}) \to \CH_0(S)$ gives a factorization of $p : \cP_{(S,H)}(\wt{\cC} / \cC) \to \CH_0(S)$ by construction.

\end{proof}



From the existence of the rational map $\gamma : S^{[g]} \dto  \ol{\cP}$ and Theorem~\ref{RCbase}, one now deduces

\begin{cor}\label{maincor}
$\CH_0\(\ol{\cP}\)$ is isomorphic to $\bZ$.
\end{cor} 

\begin{proof}[Proof]
One has the dominant rational  map $\gamma : S^{[g]} \dto \ol{\cP}$ with $\dim S^{[g]} = 2g > 2g-1 = \dim  \ol{\cP}$. Hence $ \ol{\cP}$ is rationally connected by Theorem~\ref{RCbase}, so $\CH_0(\ol{\cP}) \simeq \bZ$.

\end{proof}

\begin{cor}\label{zero}
The morphism $\Gamma_* : \CH_0(S^{[g]}) \to \CH_0(S)$ is identically zero.
\end{cor}
\begin{proof}[Proof]


Using Lemma~\ref{fact2}, $\Gamma_*$ factors through $\CH_0\(\ol{\cP}\)$. As $\CH_0\(\ol{\cP}\) \simeq \bZ$ by Corollary~\ref{maincor}, $\Gamma_*[z]$ is independent of $z \in S^{[g]}$ where $[z]$ again denotes the class of $z$ in $\CH_0\(S^{[g]}\)$. By choosing for $z$ an $\imath$-invariant zero-cycle, we conclude that $\Gamma_{*}[z] = 0$ for any $z$.

\end{proof}

\begin{proof}[Proof of Theorem~\ref{syminv}]
Corollary~\ref{zero} implies Theorem~\ref{syminv}  by the following factorization
\begin{equation*}
\begin{tikzpicture}
\centering
\matrix (m) [matrix of math nodes, row sep=1.5em,
column sep=1.5em, text height=1.5ex, text depth=0.25ex]
{  S^{[g]} &   \CH_0(S) \\
 \CH_0(S^{[g]}) & \CH_0 (\ol{\cP})     \\};
\path[ ->,font=\scriptsize]
(m-1-1) edge node[auto] {$\Gamma_*$} (m-1-2);
\path[ ->,font=\scriptsize]
(m-1-1) edge node[auto] {} (m-2-1)
(m-2-1) edge node[auto] {${\gamma_S}_*$} (m-2-2)
(m-2-2) edge node[auto] {} (m-1-2);
\end{tikzpicture}
\end{equation*}
using the fact that for a point $z$ of $S^{[g]}$ corresponding to a subscheme $Z$ of $S$, the classes $[z] \in \CH_0(S^{[g]})$ and $[Z] \in \CH_0(S)$ satisfy $\Gamma_*[z] = [Z] - \imath_*[Z]$.


\end{proof}

\begin{rem}
It is tempting to ask whether one could apply this method to symplectic automorphisms $\sigma$ of arbitrary finite order $d>2$ instead of symplectic involutions. Unfortunately, this method fails to generalize for dimension reasons. Indeed, choosing a very ample line bundle $H$ on the quotient $K3$ surface $S / (f)$ such that $c_1(H)^2 =2g-2$, then exactly as above, for a general point $s=(s_1,\ldots,s_g) \in S^g$ there exists a unique smooth curve $C_s \in |H|$ of genus $g$ such that its inverse image $\wt{C_s}$ in $S$ contains  $s_1,\ldots,s_g$. Taking 
$$\Gamma_*(s) = \sum_i\([s_i] - f_*[s_i]\),$$
one could again construct the factorization
\begin{equation*}
\begin{tikzpicture}
\centering
\matrix (m) [matrix of math nodes, row sep=1.5em,
column sep=1.5em, text height=1.5ex, text depth=0.25ex]
{  S^{[g]} &  & \CH_0(S) \\
 & \cP_{(S,H)}(\wt{\cC} / \cC) &    \\};
\path[densely dashed, ->,font=\scriptsize]
(m-1-1) edge node[below left] {$\gamma_S$} (m-2-2);
\path[ ->,font=\scriptsize]
(m-1-1) edge node[auto] {$\Gamma_*$} (m-1-3)
(m-2-2) edge node[auto] {$p$} (m-1-3);
\end{tikzpicture}
\end{equation*}
with this time, the fiber of the universal Prym variety $\cP_{(S,H)}(\wt{\cC} / \cC) \to U$ over a smooth curve $C \in U$ is the Prym variety of the étale cyclic covering $\pi : \wt{C} \to C$ induced by the quotient map $S \to S / (f)$ and is defined by
$$ \Prym_{\wt{C}/C} = \Ima\(\Id - \sigma_*\) = \Ker\(\Id + \sigma_* + \cdots + \sigma_*^{d-1} \)^\circ,$$ 
where $\sigma_* : J(\wt{C}) \to J(\wt{C})$ is the induced map on the Jacobian. Hence 
$$\dim \cP_{(S,H)}(\wt{\cC} / \cC) = (d-1)(g-1) +g \ge 2g = \dim S^{[g]},$$
so we cannot apply Theorem~\ref{RCbase}. 

\end{rem}

\bibliographystyle{alpha}
\bibliography{syminv}

\end{document}